\documentclass[11pt]{amsart}
\usepackage{amssymb}
\usepackage[english]{babel}
\usepackage{float}
\usepackage{rotating}
\usepackage{booktabs}   % for vertical spacing in a tabular
\usepackage{latexsym}
\usepackage{amsmath, amsfonts}
\usepackage{bbm}
\usepackage{amsthm,amssymb,fancyhdr,verbatim,graphicx}
\usepackage[usenames,dvipsnames]{color}
\definecolor{darkred}{rgb}{0.4,0.1,0.1}
\definecolor{darkblue}{rgb}{0.1,0.1,0.4}
\usepackage[colorlinks=true,linkcolor =darkred, citecolor=darkblue]{hyperref}
\theoremstyle{plain}% default
\newtheorem{hyp}{Hypothesis}[section]
\newtheorem{thm}{Theorem}[section]

\newtheorem*{question}{Problem}

\newtheorem*{thm*}{Theorem}

\newtheorem{lem}[thm]{Lemma}

\newtheorem{prop}[thm]{Proposition}
\newtheorem{cor}[thm]{Corollary}
\newtheorem*{cor*}{Corollary}

\newtheorem{dfn}[thm]{Definition}
\theoremstyle{remark}
\newtheorem{remark}[thm]{Remark}
\theoremstyle{plain}

\newcommand{\sign}{{\rm sign\,}}

\newcommand{\be}{\begin{equation}}
\newcommand{\ee}{\end{equation}}
\newcommand{\beu}{\begin{equation*}}
\newcommand{\eeu}{\end{equation*}}
\newcommand{\besu}{\begin{equation*}
\begin{aligned}}
\newcommand{\eesu}{\end{aligned}
\end{equation*}}
\newcommand{\bes}{\begin{equation}
\begin{aligned}}

\newcommand{\ees}{\end{aligned}
\end{equation}}

\newcommand\cP{\mathcal P}

\newcommand\frc{\mathfrak c}

\renewcommand\frq{\mathfrak q}

\newcommand\ov{\overline}

\newcommand\void[1]{}

\def\frt{{\mathfrak t}}

      \def\dC{{\mathbb C}}

\def\dD{{\mathbb D}}

   \def\dN{{\mathbb N}}   

      \def\dR{{\mathbb R}}

   \def\cE{{\mathcal E}}

\def\cP{{\mathcal P}}

\newcommand{\dom}{\mathrm{dom}\,}

\newcommand{\mes}{\mathrm{mes}\,}

\numberwithin{equation}{section}

\title[Bounds on the norms of extension operators]{Lower bounds on the norms of extension operators for Lipschitz domains}

\author{Vladimir Lotoreichik}
\address{Technische Universit\"{a}t Graz,
Institut f\"{u}r Numerische Mathematik\\
Steyrergasse 30,
8010 Graz, Austria}
\email{lotoreichik@math.tugraz.at}

\begin{document}

\maketitle

%\today

\begin{abstract}
Let $\Omega\subset\dR^d$ be a bounded or an unbounded Lipschitz domain.
In this note we address the problem of continuation of functions from the Sobolev space $H^1(\Omega)$ 
up to functions in the Sobolev space $H^1(\dR^d)$ via a linear operator. The minimal possible norm of such an operator is estimated from below in terms 
of spectral properties of self-adjoint Robin Laplacians on domains 
$\Omega$ and $\dR^d\setminus\ov\Omega$.  
Another estimate of this norm is also given, where spectral properties of Schr\"odinger operators with the $\delta$-interaction
supported on the hypersurface $\partial\Omega$ are involved. 
General results are illustrated with examples.
\end{abstract}

\section{Introduction}

Let $\Omega\subseteq\dR^d$, $d\ge 2$, be a domain with the Lebesgue space $L^2(\Omega)$ and the Sobolev space $H^1(\Omega)$ defined in the usual way.
We use the following standard norm on $H^1(\Omega)$
\[
\|f\|_{1}^2 := \|\nabla f\|^2_{L^2(\Omega;\dC^d)}+\|f\|^2_{L^2(\Omega)}.
\]
A linear extension operator is defined below.
\begin{dfn}
\label{dfn:extension}
A linear operator $E \colon H^1(\Omega)\rightarrow H^1(\dR^d)$ satisfying
the conditions
\begin{equation}
\label{extension}
(E f)|_{\Omega} = f\quad\text{and}\quad \|E\|_{1} := \sup_{\|f\|_{1} = 1}\|Ef\|_{1} <\infty
\end{equation}
is called an extension operator.
\end{dfn}
The operator $E$ provides a continuation of  functions in  $H^1(\Omega)$ up to  functions in $H^1(\dR^d)$. 
An extension operator can be constructed, in particular, for any bounded Lipschitz domain or for a hypograph
of a uniformly Lipschitz function. There are different constructions known. For hypographs
one can do a ''crude'' construction as in \cite[Theorem A.1]{McLean} via reflection of the function
with respect to the boundary. For bounded Lipschitz domains first construction is due to Calder\'{o}n \cite{C61}. 
More involved construction is due to Stein \cite{S71}, which 
has some additional important properties. 
In the case of bounded $C^\infty$-smooth domains there is a simpler construction due to Seeley \cite{S64}. 
Note that for some domains with cusps at the boundary it can be shown that no extension operator exists \cite[Exercise A.4]{McLean}.

For particular extension operators one can compute or estimate $\|E\|_1$. 
Clearly, one can claim that $\|E\|_1\in(1,+\infty)$. 
It turns out that the restriction $(Ef)|_{\Omega} = f$ does not allow 
to construct the extension operator $E$ with the norm as close to 1 as one wants.
It was probably Mikhlin who first posed the following problem.
\begin{question}
Given a domain $\Omega$ for which there exist some extension operators. Compute
\begin{equation}
\label{alpha1}
\cE(\Omega) := \inf_{E} \|E\|_1,
\end{equation}
where the infinum is taken over all extension operators.
\end{question}
In \cite{M78, M79} Mikhlin proposed an algorithm for finding $\cE(\Omega)$.
He computed $\cE(\dD)$ exactly with the aid of this algorithm, where $\dD = \{(x,y)\in\dR^2\colon x^2+y^2 < 1\}$. For general domains 
it seems to be impossible to compute $\cE(\Omega)$ and only estimates can be provided. In this paper we are interested in the lower bounds on $\cE(\Omega)$. For a bounded Lipschitz domain 
$\Omega\subset\dR^d$ the estimate 
\begin{equation}
\label{classical}
\cE(\Omega) \ge \sqrt{\frac{{\rm cap}\,\Omega}{\mes \Omega}}
\end{equation}
is known, where ${\rm cap}\, \Omega$ is the capacity of $\Omega$. This estimate is used by Maz'ya and Poborchii in \cite{MP96} and by Kalyabin in \cite{K99}. Also it is formulated as an independent statement in \cite[Lemma 3.5]{B99}. 
We present a way to estimate the value $\cE(\Omega)$ for a bounded or  an unbounded domain $\Omega$ without employing capacity. Instead we use
the knowledge of the spectra of self-adjoint Robin Laplacians and of the spectra of 
self-adjoint Schr\"odinger operators with $\delta$-interactions. 
This connection has not appeared before in an explicit form, and it might be of some interest also taking into account recent progress on Robin Laplacians, see \cite{AM12, AW03, GM09, LP08} and the references therein, and recent progress on Schr\"odinger operators with $\delta$-interactions, see the review paper \cite{E08} and also \cite{BEL13, BLL13}. Using this connection we give two different estimates of the value $\cE(\Omega)$ from below. However we do not claim that our estimates are always sharper than \eqref{classical} in the case of bounded domains.

So let $\Omega$ be a bounded or an unbounded Lipschitz domain as in Definition~\ref{def:Lipschitz}.
Consider the densely defined symmetric sesquilinear form
\[
\frt^\Omega_\beta[f,g] := (\nabla f,\nabla g)_{L^2(\Omega;\dC^d)} - \beta(f|_{\partial\Omega},g|_{\partial\Omega})_{L^2(\partial\Omega)},\qquad \dom\frt^{\Omega}_\beta := H^1(\Omega),
\]
with $\beta \ge 0$.
The form  $\frt^\Omega_\beta$ is closed and lower-semibounded \cite{AW03,AM12}. 
Therefore, it  induces by the first representation theorem
a self-adjoint operator in the Hilbert space $L^2(\Omega)$ denoted by $-\Delta^\Omega_\beta$ and usually called Robin Laplacian. 
We define the function 
\[
F_{\Omega}(\beta) := \inf\sigma(-\Delta^\Omega_\beta).
\]
It turns out that the equation
\[
F_{\Omega}(\beta) = -1
\]
has a unique strictly positive solution, 
which we denote by $\beta(\Omega)$.
Let the value $\beta(\dR^d\setminus\ov\Omega) > 0$ be a unique solution of the complementary equation
\[
F_{\dR^d\setminus\ov\Omega}(\beta) = -1
\]
corresponding to the self-adjoint Robin Laplacian acting on $\dR^d\setminus\ov\Omega$.
We prove that
\begin{equation}
\label{est1}
\cE(\Omega) \ge \sqrt{1 + \frac{\beta(\dR^d\setminus\ov\Omega)}{\beta(\Omega)}}
\end{equation}
with $\cE(\Omega)$ as in \eqref{alpha1}.  
This result implies, in particular, that
\[
\max\{\cE(\Omega),\cE(\dR^d\setminus\ov\Omega)\} \ge \sqrt{2}.
\]
Now consider another densely defined symmetric sesquilinear form
\[
\frt_{\alpha}^{\partial\Omega}[f,g] := (\nabla f,\nabla g)_{L^2(\dR^d;\dC^d)} - \alpha(f|_{\partial\Omega},g|_{\partial\Omega})_{L^2(\partial\Omega)},
\qquad \dom\frt_{\alpha}^{\partial\Omega} := H^1(\dR^d),
\]
with $\alpha \ge 0$.
The form  $\frt_{\alpha}^{\partial\Omega}$ is closed and lower-semibounded \cite{BEL13, BEKS94}. 
Therefore, it induces by the first representation theorem
a self-adjoint operator in the Hilbert space $L^2(\dR^d)$ denoted by $-\Delta_{\alpha}^{\partial\Omega}$ and usually called  the Schr\"odinger operator with $\delta$-interaction. We define the function 
\[
F_{\partial\Omega}(\alpha) := \inf\sigma(-\Delta_{\alpha}^{\partial\Omega}).
\]
It turns out as well that the equation
\[
F_{\partial\Omega}(\alpha) = -1
\]
has a unique strictly positive solution denoted by $\alpha(\partial\Omega)$ and the estimate
\begin{equation}
\label{est2}
\cE(\Omega) \ge \sqrt\frac{\alpha(\partial\Omega)}{\beta(\Omega)}
\end{equation}
holds. 

The bounds, which are proved in this paper,
work for bounded and unbounded domains simultaneously in the same form. In some cases the spectra of Robin Laplacians and $\delta$-operators are known explicitly or can be easily estimated, as we will see in our examples, that implies also explicit lower bounds on the smallest
possible norm of extension operators.

Our general results are applied to the extension problems for the wedge-type domain  $\Omega_\varphi\subset\dR^2$ with angle $\varphi\in(0,\pi]$
and for the rectangle $\Pi_{a,b} = (0,a)\times(0,b)$ with $a,b  >0$.
In these examples elementary computations lead to explicit estimates of 
$\cE(\Omega_\varphi)$ and $\cE(\Pi_{a,b})$ from below.
For the wedges we get
\begin{equation*}
\label{estimates}
\cE(\Omega_\varphi)\ge\sqrt{1+ \frac{1}{\sin(\varphi/2)}}.
\end{equation*}
For the rectangles we obtain
\begin{equation*}
\label{estimates}
\cE(\Pi_{a,b})\ge\sqrt{1+2\Big(\frac{1}{a} +\frac{1}{b}\Big)}.
\end{equation*}
In the case of wedges we also compare our lower bound with an upper bound on $\cE(\Omega_\varphi)$, which is obtained via estimation of the norm of the reflection
operator. In particular, we get the approximate asymptotic behaviour
\begin{equation}
\label{asymptotics}
\sqrt{2}+\frac{\sqrt{2}}{32}\theta^2 + o(\theta^2)\le\cE(\Omega_{\pi - \theta}) \le \sqrt{2} + \frac{\sqrt{2}}{4}\theta + o(\theta), \qquad \theta\rightarrow 0+.
\end{equation}
Previously known results on the asymptotics of the value $\cE(\Omega)$
were proved in the case that some metric parameter of the domain tends to zero. Namely, Mikhlin obtained in \cite{M81} asymptotic
behaviour of $\cE(B_r)$ for a ball $B_r$ of a radius $r > 0$ in the limit $r \rightarrow 0+$. Maz'ya and Poborchiy \cite{MP96} considered asymptotic behaviour of  $\cE(\Omega_\omega)$ for cylinder-type domain $\Omega_\omega$ with a cross section $\omega$ in the limit, when $\omega$ shrinks to a point. Kalyabin  \cite{K99} estimated $\cE(\Omega)$ for planar convex domains in the limit ${\rm diam}\,\Omega\rightarrow 0+$.  In all these papers
the value $\cE(\Omega)$ tends to $+\infty$ as the parameter tends to zero
and upper and lower bounds have the same order of growth. 
Our asymptotics \eqref{asymptotics} for wedges  is of a slightly different nature. It can be computed that $\cE(\Omega_\pi)= \sqrt{2}$ and we show how the minimal possible norm of the extension operator changes under
small deformation of the boundary of the domain. 

It remains to outline the structure of the paper. In Section~\ref{sec:prelim} we introduce self-adjoint Robin Laplacians on Lipschitz domains and self-adjoint Schr\"odinger operators with $\delta$-interactions on Lipschitz hypersurfaces and prove some of their basic properties. In Section~\ref{sec:main} we obtain our main results on the estimation of the value $\cE(\Omega)$. Section~\ref{sec:examples} is devoted to examples: in Subsection~\ref{ssec:example1} we consider an example with wedges and in Subsection~\ref{ssec:example2} we give an example with rectangles.
\subsection*{Acknowledgements}
The author gratefully acknowledges financial supported by the Austrian Science Fund (FWF): project P 25162-N26, and thanks Christian K\"uhn for the substantial help with the manuscript.
\section{Preliminaries}
\label{sec:prelim}
In this section we define a class of Lipschitz domains. Further, we introduce self-adjoint Robin Laplacians on these domains and self-adjoint Schr\"odinger operators with $\delta$-interactions supported on manifolds, which separate the Euclidean space into two such Lipschitz domains. Certain elementary spectral properties of these self-adjoint operators are proved.
\subsection{Lipschitz domains}
First we define special Lipschitz domains.
\begin{dfn}
\label{dfn:special}
The domain $\Omega\subset\dR^d$ with $d\ge 2$ is called special Lipschitz domain if there exist a coordinate system and a uniformly Lipschitz function\footnote{There exists $L > 0$ such that for any $x,y\in\dR^{d-1}$ the condition $|\varphi(x) - \varphi(y)|\le L \|x - y\|$ holds} $\varphi\colon \dR^{d-1} \rightarrow\dR$ such that in
this coordinate system
\begin{equation}
\label{special}
\Omega = \big\{(x,t)\colon x\in\dR^{d-1}, t > \varphi(x)\big\}.
\end{equation}
\end{dfn}
\noindent Throughout the paper we deal with a class of domains with Lipschitz boundary as in \cite[\S VI.3]{S71}.
\begin{dfn}
\label{def:Lipschitz}
The domain $\Omega\subset\dR^d$ is called Lipschitz domain or 
minimally smooth domain
if there exist $\varepsilon >0 $, a natural number $N$, a constant $M >0$ and a  countable family $\{U_j\}_j$ of open sets
such that:
\begin{itemize}\setlength{\itemsep}{1.2ex}
\item[\rm (i)] if $x\in\partial\Omega$, then $B_\varepsilon(x) \subset U_j$ for some $j$;
$B_\varepsilon(x)$ is the ball in $\dR^d$ with the center $x$ and the radius $\varepsilon > 0$;
\item [\rm (ii)] at most $N$ of the $U_j$'s have nonempty intersection;
\item [\rm (iii)] for each $j$ there exists a special Lipschitz domain $\Omega_j$ such that
$U_j\cap\Omega = U_j\cap\Omega_j$ and $\|\nabla \varphi_j\|_{L^\infty} \le M$, 
where $\Omega_j$ is defined by $\varphi_j$ as in \eqref{special}.
\end{itemize}
\end{dfn}
\noindent The Sobolev space $H^1(\Omega)$ is defined as usual see \cite[Chapter 3]{McLean} with the norm
\[
\|f\|_{1}^2 := \|\nabla f\|^2_{L^2(\Omega;\dC^d)} + \|f\|^2_{L^2(\Omega)}.
\]
It is known that for any $f\in H^1(\Omega)$ its trace $f|_{\partial\Omega}$  
is well-defined as a function in $L^2(\partial\Omega)$, see \cite{D96, Co88,M87}. 
\subsection{Self-adjoint Robin Laplacians}
We start with a standard statement on the Neumann sesquilinear form.
\begin{lem}\cite[\S VII.1.2]{EE}.
The densely defined, non-negative symmetric sesquilinear form
\[
\frt^\Omega_{\rm N}[f,g] := (\nabla f,\nabla g)_{L^2(\Omega;\dC^d)},\qquad \dom\frt^\Omega_{\rm N} := H^1(\Omega),
\]
is closed. 
\end{lem}
\noindent Consider the perturbation of the form $\frt^\Omega_{\rm N}$ living on the boundary 
\begin{equation}
\label{frt}
\frt^\Omega_\beta[f,g] := (\nabla f,\nabla g)_{L^2(\Omega;\dC^d)} - \beta(f|_{\partial\Omega},g|_{\partial\Omega})_{L^2(\partial\Omega)},\quad \dom\frt^\Omega_\beta := H^1(\Omega),
\end{equation}
with $\beta\in\dR$. Note that $\frt^\Omega_0 = \frt^\Omega_{\rm N}$.
The form $\frt^\Omega_\beta$ is already known to be closed and lower-semi\-bounded
for bounded Lipschitz domains, see \cite[Theorem 3.3, Proposition 4.1]{AW03} and also \cite{AM12}. In the next theorem we prove this fact for domains precisely as in Definition~\ref{def:Lipschitz}. 
This result is, of course, expected and for us it is only
an auxiliary fact.  Our proof of this fact uses a result contained in \cite{BEL13}, which is proved with the aid of Stein's extension operator.
\begin{lem}\cite{BEL13}
\label{lem:BEL}
Let $\Omega\subset\dR^d$ be a bounded or unbounded Lipschitz domain as in Definition~\ref{def:Lipschitz}.  Then for any $\varepsilon >0$ 
there exists a constant $C(\varepsilon) > 0$ such that
\[
\|f|_{\partial\Omega}\|_{L^2(\partial\Omega)}^2 \le \varepsilon \|\nabla f\|_{L^2(\Omega;\dC^d)}^2 + C(\varepsilon)\|f\|^2_{L^2(\Omega)}
\]
holds for all $f\in H^1(\Omega)$.
\end{lem}
\begin{prop}
\label{frtclosed}
The symmetric, densely defined sesquilinear form $\frt^\Omega_\beta$ from \eqref{frt} 
is closed and lower-semibounded in $L^2(\Omega)$ for all $\beta\in\dR$.
\end{prop}
\noindent We denote by $-\Delta_\beta^\Omega$ the self-adjoint operator in $L^2(\Omega)$ corresponding to the sesquilinear form $\frt^\Omega_\beta$ via the first representation theorem \cite[Chapter VI, Theorem 2.1]{Kato}.
\begin{proof}[Proof of Proposition~\ref{frtclosed}]
By Lemma~\ref{lem:BEL} the symmetric sesquilinear form 
\[
\frt'_\beta[f,g] := -\beta(f|_{\partial\Omega},g|_{\partial\Omega})_{L^2(\partial\Omega)},\qquad \dom \frt'_\beta := H^1(\Omega),
\]
is bounded with respect to the form $\frt^\Omega_{\rm N}$ with arbitrarily small form bound. Hence, by \cite[Chapter VI, Theorem 1.33]{Kato} the symmetric densely defined sesquilinear from $\frt^\Omega_\beta = \frt^\Omega_{\rm N} +\frt'_\beta$ is closed and lower-semibounded.
\end{proof}
\noindent Define for $\beta \ge 0$ the function 
\begin{equation}
\label{fOmega}
F_{\Omega}(\beta) := \inf\sigma(-\Delta_\beta^\Omega) = \inf_{\begin{smallmatrix} f\in H^1(\Omega)\\
\|f\|_{L^2(\Omega)}=1\end{smallmatrix}} \frt_\beta^\Omega[f,f].
\end{equation}
In the next proposition we collect some properties of the function $F_{\Omega}$.
\begin{prop}
\label{prop:f}
Let  $\Omega$ be a bounded or unbounded Lipschitz domain as in Definition~\ref{def:Lipschitz}.
Let the function $F_\Omega\colon[0,+\infty)\rightarrow \dR$ be defined as in \eqref{fOmega}.
\begin{itemize}\setlength{\parskip}{1.2mm}
\item [\rm (i)] $F_{\Omega}$ is non-increasing. Moreover, if $F_{\Omega}(\beta_1) < 0$, then for any $\beta_2 > \beta_1$ the strict inequality $F_{\Omega}(\beta_2) < F_{\Omega}(\beta_1)$ holds.
\item [\rm (ii)] $F_{\Omega}$ is continuous.
\item [\rm (iii)] $F_{\Omega}(0) \ge 0$.
\item[\rm (iv)]$\lim_{\beta\rightarrow+\infty} F_{\Omega}(\beta) = -\infty$.
\end{itemize}
\end{prop}
\begin{proof}
\noindent (i) Choose parameters $\beta_2 > \beta_1 \ge 0$. The relation
\[
F_{\Omega}(\beta_2) \le F_{\Omega}(\beta_1)
\]
holds straightforwardly in view of definition \eqref{fOmega}.
It remains to show the strict inequality under the assumption $F_\Omega(\beta_1) < 0$.
Suppose that $F_\Omega(\beta_1) < 0$. 
Then for any $\varepsilon >0$ there exists $f_\varepsilon\in H^1(\Omega)$ with $\|f_\varepsilon\|_{L^2(\Omega)} =1$ such that
\[
\frt^\Omega_{\beta_1}[f_\varepsilon,f_\varepsilon] \le F_\Omega(\beta_1) +\varepsilon.
\]
That implies the estimate
\[
\|f_\varepsilon|_{\partial\Omega}\|^2_{L^2(\partial\Omega)} \ge -\frac{ F_\Omega(\beta_1) +\varepsilon}{\beta_1}.
\]
Therefore, we arrive at
\[
\begin{split}
\frt^\Omega_{\beta_2}[f_\varepsilon,f_\varepsilon] &= \frt_{\beta_1}^\Omega[f_\varepsilon, f_\varepsilon] + (\beta_1 -\beta_2)\|f_\varepsilon|_{\partial\Omega}\|_{L^2(\partial\Omega)}^2\\
&\le \big( F_\Omega(\beta_1)+\varepsilon\big)\big(1 + \tfrac{\beta_2 - \beta_1}{\beta_1}\big).
\end{split}
\]
As a result by the choice of sufficiently small $\varepsilon > 0$ we obtain
\[
F_\Omega(\beta_2) \le \frt^\Omega_{\beta_2}[f_\varepsilon,f_\varepsilon] < F_\Omega(\beta_1).
\]
(ii) 
Let us define left and right limits of $F_\Omega(\cdot)$
\[
F_{\Omega,-}(\beta) := \lim_{x\rightarrow \beta-} F_\Omega(x)\quad\text{and}\quad F_{\Omega,+}(\beta) := \lim_{x\rightarrow \beta+} F_\Omega(x),
\]
which are well-defined because of monotonicity of $F_\Omega$. Then we carry out the proof of this item into two steps.\\
\noindent \emph{Step I:}
Suppose that for some $\beta_0 >0$ the inequality
\begin{equation}
\label{ljump}
F_{\Omega,-}(\beta_0) > F_\Omega(\beta_0)
\end{equation}
holds. Let us fix sufficiently small value $\varepsilon >0$ and choose a function $f_\varepsilon\in H^1(\Omega)$ with $\|f_\varepsilon\|_{L^2(\Omega)}=  1$ such that
\[
\frt_{\beta_0}^\Omega[f_\varepsilon,f_\varepsilon] \le F_\Omega(\beta_0) + \varepsilon. 
\] 
Substituting this function into the form $\frt_x^\Omega$ with $x < \beta_0$ we get
\[
\begin{split}
F_\Omega(x) &\le \frt_x^\Omega[f_\varepsilon,f_\varepsilon] = \frt_\beta^\Omega[f_\varepsilon,f_\varepsilon] + (\beta_0 -x)\|f_\varepsilon|_{\partial\Omega}\|^2_{L^2(\partial\Omega)} \\
&\le  F_\Omega(\beta_0) + \varepsilon+
(\beta_0 -x)\|f_\varepsilon|_{\partial\Omega}\|^2_{L^2(\partial\Omega)}.
\end{split}
\]
Passing to the limit $x\rightarrow \beta_0-$ we get
\[
F_{\Omega,-}(\beta_0) \le F_\Omega(\beta_0) + \varepsilon,
\]
but that contradicts with \eqref{ljump}.\\
\noindent \emph{Step II:}
Suppose that for some $\beta_0 \ge  0$ the inequality
\begin{equation}
\label{rjump}
F_\Omega(\beta_0) > F_{\Omega,+}(\beta_0)  
\end{equation}
holds. Let us fix arbitrarily small $\varepsilon > 0$. Note that due to Lemma~\ref{lem:BEL} we get the estimate
\[
\big|\frt^\Omega_{\beta_0}[f,f]\big| \ge (1 -\beta_0\varepsilon)\|\nabla f\|^2_{L^2(\Omega;\dC^d)} - \beta_0C(\varepsilon)\|f\|^2_{L^2(\Omega)},
\]
which is equivalent to
\begin{equation}
\label{estimate}
\|\nabla f\|^2_{L^2(\Omega;\dC^d)} \le \frac{1}{1-\beta_0\varepsilon}\big|\frt^\Omega_{\beta_0}[f,f]\big|+ \frac{\beta_0C(\varepsilon)}{1-\beta_0\varepsilon}\|f\|^2_{L^2(\Omega)}
\end{equation}
Choose a sequence of non-negative values $\{\beta_n\}$ such that $\beta_n\rightarrow \beta_0+$ mono\-tonously. According to \eqref{estimate} and using Lemma~\ref{lem:BEL} the forms $\frt_{\beta_n}^\Omega$  and  $\frt_{\beta_0}^\Omega$ satisfy the estimate
\[
\begin{split}
\Big|\frt_{\beta_n}^\Omega[f,f]-\frt_{\beta_0}^\Omega[f,f]\Big|& \le \varepsilon|\beta_n-\beta_0|\|\nabla f\|^2_{L^2(\Omega;\dC^d)}+ C(\varepsilon)|\beta_n-\beta_0|\|f\|^2_{L^2(\Omega)} \\
&\le \frac{\varepsilon|\beta_n-\beta_0|}{1-\beta_0\varepsilon}
\big|\frt_{\beta_0}^\Omega[f,f]\big| + \frac{|\beta_n-\beta_0|}{1-\beta_0\varepsilon}C(\varepsilon) \|f\|^2_{L^2(\Omega)}.
\end{split}
\]
for all $f\in H^1(\Omega)$. By \cite[Theorem VI.3.6, Theorem IV.2.23]{Kato} the sequence of operators $-\Delta_{\beta_n}^\Omega$ converges to the operator $-\Delta_{\beta_0}^\Omega$ in the norm resolvent sense. Obviously the inclusion $(-\infty,F_\Omega(\beta_0))\subset\rho(-\Delta_{\beta_0}^\Omega)$ holds. By \cite[Theorem IV.3.1]{Kato} for all $\varepsilon > 0$ there exists $N\in\dN$ such that for all $n\ge N$ the inclusion
\[
(-\infty, F_\Omega(\beta_0)-\varepsilon)\subset \rho(-\Delta_{\beta_n}^\Omega).
\]
holds. By choosing $\varepsilon >0 $ such that $F_\Omega(\beta_0) -\varepsilon > F_{\rm \Omega,+}(\beta_0)$ we arrive at a contradiction with \eqref{rjump}.
Hence, concluding both steps for all $\beta\in\dR_+$ the equality
\[
F_\Omega(\beta) = F_{\Omega,\pm}(\beta)
\]
holds, and the function $F_\Omega$ is continuous.\\
\noindent (iii) 
In view of the definition of $F_\Omega$ in \eqref{fOmega} we may conclude that $F_\Omega(0) \ge 0$. \\
\noindent (iv)
Clearly, there exists a function $f\in H^1(\Omega)$ such that $\|f\|_{L^2(\Omega)} =1$ and that $\|f|_{\partial\Omega}\|_{L^2(\partial\Omega)}^2 >0$, that yields
\[
\lim_{\beta\rightarrow +\infty}F_\Omega(\beta) \le \lim_{\beta\rightarrow +\infty} \frt^\Omega_\beta[f,f] = -\infty.
\]
\end{proof}
\begin{remark}The results of Proposition~\ref{prop:f} for bounded Lipschitz domains follow from \cite[Proposition 3]{AM12}. We require this fact also for unbounded domains and for this reason the full proof is provided.
\end{remark}
\begin{cor}
\label{cor:f}
Let $F_\Omega$ be as in \eqref{fOmega}. Then the equation 
\begin{equation}
\label{eqf}
F_\Omega(\beta) = -1
\end{equation}
on $\beta$ has a unique strictly positive solution denoted by  $\beta(\Omega)$.
\end{cor}
\begin{proof}
The existence and strict positivity of the solution follow from $F_\Omega(0) \ge 0$, $\lim_{\beta\rightarrow+\infty}F_\Omega(\beta) =-\infty$ and from the continuity of $F_\Omega$.
Uniqueness of the solution follows from the monotonicity properties of $F_\Omega$.
\end{proof}

\subsection{Self-adjoint Schr\"odinger operators with $\delta$-interactions}
Let $\Sigma\subset\dR^d$  be a $(d-1)$-dimensional manifold, which separates the Euclidean space $\dR^d$ into two Lipschitz domains $\Omega$ and $\dR^d\setminus\ov\Omega$, where $\partial\Omega = \Sigma$. We shall consider
the symmetric, densely defined sesquilinear form
\[
\frt_{\alpha}^{\partial\Omega}[f,g] := (\nabla f, \nabla g)_{L^2(\dR^d;\dC^d)} -\alpha(f|_{\partial\Omega},g|_{\partial\Omega})_{L^2(\Sigma)},\qquad \dom\frt_{\alpha}^{\partial\Omega} = H^1(\dR^d).
\]
According to \cite{BEL13, BEKS94} the form $\frt_{\alpha}^{\partial\Omega}$ is closed and lower-semibounded. The self-adjoint operator corresponding to the form $\frt_{\alpha}^{\partial\Omega}$ via the first representation theorem will be denoted by $-\Delta_{\alpha}^{\partial\Omega}$. Let us introduce the 
function
\begin{equation}
\label{fSigma}
F_{\partial\Omega}(\alpha) :=\inf\sigma(-\Delta_\alpha^{\partial\Omega}) :=\inf_{\begin{smallmatrix} f\in H^1(\dR^d)\\ \|f\|_{L^2(\dR^d)}=1\end{smallmatrix}} \frt_\alpha^{\partial\Omega}[f,f].
\end{equation}
Next we will formulate without proofs complete analogs of Proposition~\ref{prop:f} and Corollary~\ref{cor:f} for the function $F_{\partial\Omega}$.
\begin{prop}
\label{prop:f}
Let  $\Omega\subset\dR^d$ be a bounded or unbounded Lipschitz domain as in Definition~\ref{def:Lipschitz}.
Let the function $F_{{\partial\Omega}}\colon\dR_+\rightarrow \dR$ be defined as in \eqref{fOmega}.
\begin{itemize}\setlength{\parskip}{1.2mm}
\item [\rm (i)] $F_{\partial\Omega}$ is non-increasing. Moreover, if $F_{\partial\Omega}(\alpha_1) < 0$, then for any $\alpha_2 > \alpha_1$ the strict inequality $F_{\partial\Omega}(\alpha_2) < F_{\partial\Omega}(\alpha_1)$ holds.
\item [\rm (ii)] $F_{\partial\Omega}$ is continuous.
\item [\rm (iii)] $F_{\partial\Omega}(0) = 0$.
\item[\rm (iv)]$\lim_{\alpha\rightarrow+\infty} F_{\partial\Omega}(\alpha) = -\infty$.
\end{itemize}
\end{prop}
\begin{cor}
Let $F_{\partial\Omega}$ be as in \eqref{fSigma}. Then the equation 
\begin{equation}
\label{eqf2}
F_{\partial\Omega}(\alpha) = -1
\end{equation}
on $\alpha$ has a unique strictly positive solution denoted by $\alpha(\partial\Omega)$.
\end{cor}

\section{Lower bounds on the norms of extension operators}
\label{sec:main}
In this section we estimate the value $\cE(\Omega)$ from below using the unique roots
of the equations $F_\Omega(\beta) = -1$, $F_{\dR^d\setminus\ov{\Omega}}(\beta) = -1$
and $F_{\partial\Omega}(\alpha) = -1$. In the first lower bound on $\cE(\Omega)$ the unique roots of the equations $F_\Omega(\beta) = -1$, $F_{\dR^d\setminus\ov{\Omega}}(\beta) = -1$ are employed.
\begin{thm}
\label{thm:bnd1}
Let $\Omega\subset \dR^d$ be a bounded or unbounded Lipschitz domain as in Definition~\ref{def:Lipschitz}. Let the values $\beta(\Omega) > 0$ and $\beta(\dR^d\setminus\ov\Omega) > 0$ be defined as the solutions of the equation \eqref{eqf} for the domains $\Omega$ and $\dR^d\setminus\ov\Omega$, respectively.
Let the value $\cE(\Omega)$ be defined as in \eqref{alpha1}.
Then the following estimate 
\[
\cE(\Omega) \ge \sqrt{1 + \frac{\beta(\dR^d\setminus\ov\Omega)}{\beta(\Omega)}}
\]
holds. In particular, $\max\{\cE(\Omega), \cE(\dR^d\setminus\ov\Omega)\} \ge \sqrt{2}$.
\end{thm}
\begin{proof}
Suppose that $\beta(\Omega) > 0 $ is the unique solution of the equation $F_\Omega(\beta) = -1$. Then for any $\varepsilon >0$ there exists a function $f_\varepsilon \in  H^1(\Omega)$ such that 
\[
\frt^\Omega_{\beta(\Omega)}[f_\varepsilon,f_\varepsilon] = \|\nabla f_\varepsilon\|^2_{L^2(\Omega;\dC^d)} - \beta(\Omega)\|f_\varepsilon|_{\partial\Omega}\|_{L^2(\partial\Omega)}^2 \le (-1+\varepsilon)\|f_\varepsilon\|_{L^2(\Omega)}^2.
\]
This inequality can be rewritten in a more suitable form
\begin{equation}
\label{estimate1}
\|f_\varepsilon|_{\partial\Omega}\|_{L^2(\partial\Omega)}^2 \ge \frac{1}{\beta(\Omega)}\|\nabla  f_\varepsilon\|^2_{L^2(\Omega;\dC^d)} + \frac{1-\varepsilon}{\beta(\Omega)}\|f_\varepsilon\|_{L^2(\Omega)}^2.
\end{equation}
Suppose that $E$ is an arbitrary extension operator for the domain $\Omega$
as in Definition~\ref{dfn:extension}. Let us apply this operator to the function $f_\varepsilon$
and denote
\[
g_\varepsilon := (Ef_\varepsilon)|_{\dR^d\setminus\ov\Omega}\in H^1(\dR^d\setminus\ov\Omega).
\]
Recall that $\beta(\dR^d\setminus\ov\Omega) > 0$ is the unique solution of the equation
$F_{\dR^d\setminus\ov\Omega}(\beta) = -1$. Therefore, in view of \eqref{frt}, we get
\[
\|\nabla g_\varepsilon\|^2_{L^2(\dR^d\setminus\Omega;\dC^d)} - \beta(\dR^d\setminus\ov\Omega)
\|g_\varepsilon|_{\partial\Omega}\|^2_{L^2(\partial\Omega)} \ge -\|g_\varepsilon\|^2_{L^2(\dR^d\setminus\Omega)}.
\]
The latter can be rewritten as
\begin{equation}
\label{estimate2}
\|g_\varepsilon|_{\partial\Omega}\|^2_{L^2(\partial\Omega)}\le \frac{1}{\beta(\dR^d\setminus\ov\Omega)}\|\nabla g_\varepsilon\|^2_{L^2(\dR^d\setminus\ov\Omega;\dC^d)} + \frac{1}{\beta(\dR^d\setminus\ov\Omega)}\|g_\varepsilon\|^2_{L^2(\dR^d\setminus\ov\Omega)}.
\end{equation}
Note that by the definition of the operator $E$ and the properties of the Sobolev space $H^1(\dR^d)$ we can state that
\begin{equation}
\label{equality}
g_\varepsilon|_{\partial\Omega} = f_\varepsilon|_{\partial\Omega}.
\end{equation}
Now estimates \eqref{estimate1}, \eqref{estimate2} and observation \eqref{equality}
imply
\[
\begin{split}
&\frac{1}{\beta(\Omega)}\|\nabla  f_\varepsilon\|^2_{L^2(\Omega;\dC^d)} + \frac{1-\varepsilon}{\beta(\Omega)}\|f_\varepsilon\|_{L^2(\Omega)}^2\\
&\qquad\le \frac{1}{\beta(\dR^d\setminus\ov\Omega)}\|\nabla g_\varepsilon\|^2_{L^2(\dR^d\setminus\ov\Omega;\dC^d)} \!+\! \frac{1}{\beta(\dR^d\setminus\ov\Omega)}\|g_\varepsilon\|^2_{L^2(\dR^d\setminus\ov\Omega)},
\end{split}
\]
which leads to
\[
\begin{split}
&\|\nabla  f_\varepsilon\|^2_{L^2(\Omega;\dC^d)} + (1-\varepsilon)\|f_\varepsilon\|^2_{L^2(\Omega)} \le\\
&\qquad\qquad\le \frac{\beta(\Omega)}{\beta(\dR^d\setminus\ov\Omega)}\|\nabla g_\varepsilon\|^2_{L^2(\dR^d\setminus\ov\Omega;\dC^d)} + \frac{\beta(\Omega)}{\beta(\dR^d\setminus\ov\Omega)}\|g_\varepsilon\|^2_{L^2(\dR^d\setminus\ov\Omega)} 
\end{split}
\]
That yields
\[
(1-\varepsilon)\|f_\varepsilon\|^2_{1}\le \frac{\beta(\Omega)}{\beta(\dR^d\setminus\ov\Omega)}\|g_\varepsilon\|_{1}^2.
\]
Note that $Ef_\varepsilon = f_\varepsilon\oplus g_\varepsilon$, and hence  for all sufficiently small $\varepsilon >0$
\[
\|E\|^2_{1}\ge 1 + \big(1-\varepsilon\big)\frac{\beta(\dR^d\setminus\ov\Omega)}{\beta(\Omega)}.
\]
Passing to the limit $\varepsilon\rightarrow 0+$ we arrive at
\[
\|E\|_{1} \ge\sqrt{ 1+\frac{\beta(\dR^d\setminus\ov\Omega)}{\beta(\Omega)} }.
\]
That finishes the proof.
\end{proof}
\begin{remark}
\label{rem:sharp}
We should say a few words on the sharpness of the obtained lower bound. For the half-space 
\[
\dR^d_+ = \{(x_1,x_2,\dots,x_d)\in\dR^d\colon x_d > 0\}
\]
it is known that $F_{\dR^d_+}(\beta)  = -\beta^2$, that is $\beta(\dR^d_+) = 1$ and analogously $\beta(\dR^d_-) =1$. Hence, Theorem~\ref{thm:bnd1} yields
\[
\cE(\dR^d_+) \ge\sqrt{2}.
\]
On the other hand the reflection operator
\[
(Ef)(x) := \begin{cases} f(x_1,x_2,\dots,x_d),& \quad x_d > 0,\\
f(x_1,x_2,\dots,-x_d),&\quad x_d < 0,
\end{cases}
\]
is an extension operator in the sense of Definition~\ref{dfn:extension}
and 
\[
\|E\|_{1} = \sqrt{2}.
\]
So in fact $\cE(\dR^d_+) =\sqrt{2}$ and in this respect the bound in Theorem~\ref{thm:bnd1} is sharp.
\end{remark}
In the next theorem we obtain the lower bound on $\cE(\Omega)$ in terms of the roots of the equations $F_{\partial\Omega}(\alpha) = -1$ and $F_\Omega(\beta) = -1$.
\begin{thm}
\label{thm:bnd2}
Let $\Omega\subset \dR^d$ be a bounded or unbounded Lipschitz domain as in Definition~\ref{def:Lipschitz} with the boundary $\partial\Omega$. Let the value $\beta(\Omega) > 0$ be defined as the solution of the equation \eqref{eqf} for the domain $\Omega$, and let the value $\alpha(\partial\Omega)$ be defined as the solution of the equation \eqref{eqf2} for the hypersurface $\partial\Omega$.
Let the value $\cE(\Omega)$ be defined as in \eqref{alpha1}.
Then the following estimate 
\[
\cE(\Omega) \ge \sqrt{\frac{\alpha(\partial\Omega)}{\beta(\Omega)}}
\]
holds. 
\end{thm}
\begin{proof}
Suppose that $\beta(\Omega) > 0 $ is the unique solution of the equation $F_\Omega(\beta) = -1$. Then for any $\varepsilon >0$ there exists a function $f_\varepsilon \in  H^1(\Omega)$ such that 
\[
\frt^\Omega_{\beta(\Omega)}[f_\varepsilon,f_\varepsilon] = \|\nabla f_\varepsilon\|^2_{L^2(\Omega;\dC^d)} - \beta(\Omega)\|f_\varepsilon|_{\partial\Omega}\|_{L^2(\partial\Omega)}^2 \le (-1+\varepsilon)\|f_\varepsilon\|_{L^2(\Omega)}^2.
\]
This inequality can be rewritten in a more suitable form
\begin{equation}
\label{estimate3}
\|f_\varepsilon|_{\partial\Omega}\|_{L^2(\partial\Omega)}^2 \ge \frac{1}{\beta(\Omega)}\|\nabla  f_\varepsilon\|^2_{L^2(\Omega;\dC^d)} + \frac{1-\varepsilon}{\beta(\Omega)}\|f_\varepsilon\|_{L^2(\Omega)}^2.
\end{equation}
Suppose that $E$ is an arbitrary extension operator for the domain $\Omega$
as in Definition~\ref{dfn:extension}. Let us apply this operator to the function $f_\varepsilon$ and define 
\[
h_\varepsilon := Ef_\varepsilon\in H^1(\dR^d).
\]
Clearly, the relation 
\begin{equation}
\label{obs2}
f_\varepsilon|_{\partial\Omega} = h_\varepsilon|_{\partial\Omega}
\end{equation}
holds. Since $\alpha(\partial\Omega) > 0$ is the root of $F_{\partial\Omega}(\alpha) = -1$
we arrive at the inequality
\[
\|\nabla h_\varepsilon\|^2_{L^2(\dR^d;\dC^d)} -\alpha(\partial\Omega)\|h_\varepsilon|_{\partial\Omega}\|^2_{L^2(\partial\Omega)} \ge -\|h_\varepsilon\|^2_{L^2(\dR^d)}, 
\]
which can be rewritten as
\begin{equation}
\label{estimate4}
\alpha(\partial\Omega)\|h_\varepsilon|_{\partial\Omega}\|^2_{L^2(\partial\Omega)} \le \|\nabla h_\varepsilon\|^2_{L^2(\dR^d;\dC^d)}+\|h_\varepsilon\|^2_{L^2(\dR^d)}, 
\end{equation}
Combining the estimates \eqref{estimate3}, \eqref{estimate4} and the observation \eqref{obs2} we get
\[
\begin{split}
& \frac{\alpha(\partial\Omega)}{\beta(\Omega)}\|\nabla  f_\varepsilon\|^2_{L^2(\dR^d;\dC^d)} + \frac{\alpha(\partial\Omega)}{\beta(\Omega)}(1-\varepsilon)\|f_\varepsilon\|^2_{L^2(\dR^d)} \le\\
&\qquad\qquad\le\|\nabla h_\varepsilon\|^2_{L^2(\dR^d;\dC^d)} +\|h_\varepsilon\|^2_{L^2(\dR^d)} .
\end{split}
\]
Hence, we obtain the estimate
\[
\frac{\alpha(\partial\Omega)}{\beta(\Omega)}(1-\varepsilon)\|f_\varepsilon\|_{1}^2\le \|h_\varepsilon\|_{1}^2,
\]
which implies for all sufficiently small $\varepsilon >0$
\[
\|E\|^2_{1} \ge\frac{\alpha(\partial\Omega)}{\beta(\Omega)}(1-\varepsilon).
\]
That leads to 
\[
\|E\|_{1} \ge \sqrt{\frac{\alpha(\partial\Omega)}{\beta(\Omega)}}
\]
in the limit $\varepsilon \rightarrow 0+$ and the claim is proved.
\end{proof}
\begin{remark}
Similar argumentation as in Remark~\ref{rem:sharp} and the fact that $F_{\partial\dR^d_+}(\alpha) = -\alpha^2/4$ give sharpness of the estimate in Theorem~\ref{thm:bnd2} for half-spaces.
\end{remark}
\section{Examples}
\label{sec:examples}
We supplement our general estimates with two examples. One example with domains of infinite measure and one example with domains of finite measure.
\subsection{Extension operators on wedges}
\label{ssec:example1}
In this subsection our aim is to illustrate obtained general estimates in the case of wedges, which are domains of infinite measure.
The wedge with angle $\varphi\in(0,\pi)$ can be defined as a hypograph
\begin{equation}
\label{wedge}
 \Omega_\varphi := \big\{(x_1,x_2)\in\dR^2\colon x_2 > \cot(\varphi/2) |x_1|\big\}.
\end{equation}
That is a special Lipschitz domain as in Definition~\ref{dfn:special}, whose
boundary is defined by the function
\[
\xi(x_1)  := \cot(\varphi/2)|x_1|,
\]
see Figure~\ref{fig}.
\begin{figure}[H]
\begin{center}
\begin{picture}(200,120)
\put(100,10){\line(-1,2){50}}
\put(100,10){\line(1,2){50}}
\put(85,90){$\Omega_\varphi$}
\put(54,110){\begin{turn}{-66.6}$\Sigma_2$\end{turn}}
\put(135,105){\begin{turn}{66.6}$\Sigma_1$\end{turn}}
\qbezier(98,14)(100,16)(102,14)
\put(94,21){\small{$ \varphi$}}
\put(100,10){\vector(0,1){110}}
\put(100,10){\vector(1,0){90}}
\put(105,114){$x_2$}
\put(185,15){$x_1$}
%\put(211,12){\small{$ \varphi$}}
\end{picture}
\end{center}
\caption{A wedge $\Omega\subset\dR^2$ with angle $\varphi\in(0,\pi)$ having sides $\Sigma_1$ and $\Sigma_2$.}
\label{fig}
\end{figure}
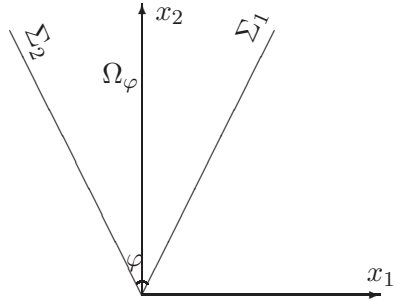
Using Theorem~\ref{thm:bnd1} and spectral results of \cite{LP08} we obtain a lower  bound on $\cE(\Omega_\varphi)$
and by means of reflection operator we obtain an upper bound on $\cE(\Omega_\varphi)$. Having a two-sided estimate of 
$\cE(\Omega_\varphi)$ we compute its asymptotic behaviour in the limit $\varphi\rightarrow \pi-$.
We make use of the following result contained in \cite{LP08}.
\begin{lem}\cite[Lemma 2.6, Lemma 2.8]{LP08}
\label{LP}
Let $\Omega_\varphi\subset\dR^2$ be a wedge with $\varphi\in(0,\pi)$.
Let the functions $F_{\Omega_\varphi}$ and $F_{\dR^2\setminus\ov{\Omega_\varphi}}$ be defined as in \eqref{fOmega}.
Then the following statements hold:
\begin{itemize}\setlength{\parskip}{1.2mm}
\item [\rm (i)] $F_{\Omega_\varphi}(\beta) = -\frac{\beta^2}{\sin^2(\varphi/2)}$;
\item [\rm (ii)] $F_{\dR^2\setminus\ov{\Omega_\varphi}}(\beta) = -\beta^2$.
\end{itemize}
\end{lem}
The lower bound on $\cE(\Omega_\varphi)$ follows from Theorem~\ref{thm:bnd1} and Lemma~\ref{LP}.
\begin{prop}
\label{Ebnd}
Let $\Omega_\varphi\subset\dR^2$ be a wedge with $\varphi\in(0,\pi)$.
Let $\cE(\Omega_\varphi)$ be defined as in \eqref{alpha1}.
Then the following estimate
\[ 
\cE(\Omega_\varphi)\ge\sqrt{1+\frac{1}{\sin(\varphi/2)}}
\]
holds. In particular, $\lim_{\varphi\rightarrow 0+} \cE(\Omega_\varphi) = +\infty$
\end{prop}
\begin{proof}
By Lemma~\ref{LP} the equation $F_{\Omega_\varphi}(\beta) = -1$
has the unique solution
\begin{equation}
\label{root1}
\beta(\Omega_\varphi) = \sin(\varphi/2),
\end{equation}
and the equation  $F_{\dR^2\setminus\ov{\Omega_\varphi}}(\beta) = -1$
has the unique solution
\begin{equation}
\label{root2}
\beta(\dR^2\setminus\ov{\Omega_\varphi}) = 1.
\end{equation}
Combining \eqref{root1}, \eqref{root2} and Theorem~\ref{thm:bnd1} we get the desired bound on $\cE(\Omega_\varphi)$.
\end{proof}
In the next lemma we give a spectral result from \cite{L13}.
\begin{lem}
\label{L}
Let $\Omega_\varphi\subset\dR^2$ be a wedge with $\varphi\in(0,\pi)$ with the boundary $\partial\Omega_\varphi$. Let the function $F_{\partial\Omega_\varphi}$ be defined as in \eqref{fSigma}. Then the following estimate
\[
F_{\partial\Omega_\varphi}(\alpha) \ge -\frac{\alpha^2}{(1+\sin(\varphi/2))^2}
\]
holds.
\end{lem}
Using Theorem~\ref{thm:bnd2}, Lemma~\ref{LP} and Lemma~\ref{L}
one gets the same lower bound on $\cE(\Omega_\varphi)$ as in Corollary~\ref{Ebnd}.
There is some hope (but we have no proof yet) that
\[
F_{\partial\Omega_\varphi}(\alpha) > -\frac{\alpha^2}{(1+\sin(\varphi/2))^2}.
\]
Then using Theorem~\ref{thm:bnd2} one gets better lower bound on $\cE(\Omega_\varphi)$ than in Proposition~\ref{Ebnd}.

In order to estimate $\cE(\Omega_\varphi)$ from above it suffices to take an arbitrary extension
operator as in Definition~\ref{dfn:extension} for the domain $\Omega_\varphi$ and estimate its
norm from above. We shall take the simplest one, which reflects the function
with respect to the boundary. The reflection with respect to the boundary of $\Omega_\varphi$ is defined as
\begin{equation}
\label{E}
\begin{split}
&E\colon H^1(\Omega_\varphi)\rightarrow H^1(\dR^2),\\
&\big(Ef\big)(x_1,x_2) := \begin{cases} f(x_1,x_2),& x_2 > \xi(x_1),\\
                  f(x_1,2\xi(x_1) - x_2),& x_2 \le  \xi(x_1).
                 \end{cases}
\end{split}
\end{equation}
In the next proposition we estimate the norm of $E$ from \eqref{E} from above.
In the special case of the wedge this estimation
is slightly finer than one can find in \cite[Appendix A]{McLean}.
\begin{prop}
\label{Ebnd2}
Let $\Omega_\varphi$ be a wedge as in \eqref{wedge}, and let the value $\cE(\Omega_\varphi)$ be defined as in \eqref{alpha1}.
Then the following estimate
\[
 \cE(\Omega_\varphi)\le \frac{\sqrt{2}}{\sin(\varphi/2)}\sqrt{1+\cos(\varphi/2)}
\]
holds.
\end{prop}
\begin{proof}
In the proof we use that the module of the Jacobian of the substitution
\[
(x_1,x_2) \mapsto (x_1, 2\xi(x_1) - x_2)
\]
is equal to one. Observe that the equality
\begin{equation}
\label{norm}
 \|Ef\|_{L^2(\dR^2)}^2=2\|f\|_{L^2(\Omega_\varphi)}^2
\end{equation}
holds, where we used that
\[
\|f\|^2_{L^2(\Omega_\varphi)} = \int_{\Omega_\varphi} |(Ef)(x_1,x_2)|^2dx_1dx_2 =\int_{\dR^2\setminus\ov{\Omega_\varphi}} 
|(Ef)(x_1,x_2)|^2dx_1dx_2.
\]
As the next step we compute directly partial derivatives of the function $Ef$:
\begin{equation*}
\label{partial1}
\begin{split}
(\partial_1 Ef)(x_1,x_2) & =  (\partial_1 f)(x_1,x_2),\quad x_2> \xi(x_1),\\
(\partial_1 Ef)(x_1,x_2) &=\big((\partial_1  +\sign(x_1)2\cot(\varphi/2)\partial_2)f\big)(x_1, 2\xi(x_1)-x_2),\quad\!\!\! x_2 \le\xi(x_1),
\end{split}
\end{equation*}
and
\begin{equation*}
\label{partial2}
(\partial_2 Ef)(x_1,x_2) = 
\begin{cases} 
(\partial_1 f)(x_1,x_2),&x_2 > \xi(x_1),\\
 (-\partial_2f)(x_1, 2\xi(x_1)-x_2), &x_2 \le \xi(x_1).
\end{cases}
\end{equation*}
Further, we estimate the norms of $\partial_1 Ef$ and  $\partial_2 Ef$
\begin{equation}
\label{partial1n}
 \begin{split}
 \|\partial_1 Ef\|_{L^2(\dR^2)}^2 &= \|\partial_1 f\|^2_{L^2(\Omega_\varphi)} + \|\partial_1f + \sign(x_1)2\cot(\varphi/2)\partial_2 f\|^2_{L^2(\Omega_\varphi)}\\
&\le (2+t)\|\partial_1 f\|^2_{L^2(\Omega_\varphi)} + 4\cot^2(\varphi/2)(1+1/t)\|\partial_2 f\|^2_{L^2(\Omega_\varphi)}
\end{split}
\end{equation}
with any $t > 0$ and
\begin{equation}
\label{partial2n}
\|\partial_2 Ef\|_{L^2(\dR^2)}^2 = 2\|\partial_2 f\|_{L^2(\Omega_\varphi)}^2.
\end{equation}
The estimates \eqref{norm}, \eqref{partial1n} and \eqref{partial2n} imply
\[
\begin{split}
\|Ef\|^2_{H^1(\dR^2)} &\le (2+t)\|\partial_1 f\|^2_{L^2(\Omega_\varphi)} \\
&\qquad +(2 + 4\cot^2(\varphi/2)(1+1/t))\|\partial_2 f\|^2_{L^2(\Omega_\varphi)} + 2\|f\|_{L^2(\Omega_\varphi)}^2.
\end{split}
\]
Hence,
\[
\|Ef\|^2_1 \le \max\{2+t,2+4\cot^2(\varphi/2)(1+1/t),2\}\|f\|_{1}^2.
\]
Optimizing the maximum between three values with respect to the parameter $t > 0$ we obtain that the maximum is minimal for 
\[
 t = 2\cot^2(\varphi/2) + \sqrt{4\cot^4(\varphi/2) +4\cot^2(\varphi/2)} = 2\cot^2(\varphi/2) + \frac{2\cot(\varphi/2)}{\sin(\varphi/2)}.
\]
That gives us the estimate
\[
\|E\|_1 \le \sqrt{2+t} = \frac{\sqrt{2}}{\sin(\varphi/2)}\sqrt{1+\cos(\varphi/2)}
\]
and, hence, by \eqref{alpha1}
\[
 \cE(\Omega_\varphi)\le \frac{\sqrt{2}}{\sin(\varphi/2)}\sqrt{1+\cos(\varphi/2)}.
\]
\end{proof}
\begin{cor}
\label{cor:asymp}
In the assumptions of the proposition above
\[
\sqrt{2}+\frac{\sqrt{2}}{32}\theta^2 + o(\theta^2)\le\cE(\Omega_{\pi - \theta}) \le \sqrt{2} + \frac{\sqrt{2}}{4}\theta + o(\theta), \qquad \theta\rightarrow 0+.
\]
\end{cor}
\begin{proof}
The expansions in this corollary follow from the bounds in Propositions~\ref{Ebnd} and \ref{Ebnd2}.
\end{proof}
\begin{remark}
Note that upper and lower bounds in Corollary~\ref{cor:asymp} have different order of convergence to $\sqrt{2}$ and exact asymptotics of $\cE(\Omega_{\pi -\theta})$ in the limit $\theta\rightarrow 0+$ remains an open problem.
\end{remark}
\begin{remark}
In \cite[Subsection 7.3]{E08} it is conjectured that
\[
F_{\partial\Omega_{\pi - \theta}}(\alpha) = -\frac{1}{4}\alpha^2 -c'\alpha^2\theta^4 + O(\theta^5),\quad \theta\rightarrow 0+,
\]
with some constant $c' > 0$.
Using this conjecture, Lemma~\ref{LP}\,(i) and Theorem~\ref{thm:bnd2} one gets
the asymptotic lower bound
\[
\cE(\Omega_{\pi -\theta}) \ge \sqrt{2} + \frac{\sqrt{2}}{16}\theta^2 + o(\theta^2),\quad \theta \rightarrow 0+,
\]
which is a bit sharper than the bound in Corollary~\ref{cor:asymp}.
\end{remark}
\subsection{Extension operators on rectangles}
\label{ssec:example2}
In this subsection our aim is to illustrate obtained general estimates in the case of rectangles, which are Lipschitz domains of finite measure.
Let
\begin{equation}
\label{rectangle}
\Pi_{a,b} = (0,a)\times (0,b)\subset \dR^2
\end{equation}
be a rectangle with
the lengths of the edges $a > 0$ and $b > 0$, respectively, see Figure~\ref{fig2}.
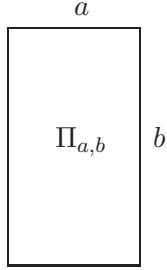
\begin{figure}[H]
\begin{center}
\begin{picture}(200,120)
\put(75,100){\line(1,0){50}}
\put(75,10){\line(1,0){50}}
\put(75,100){\line(0,-1){90}}
\put(125,100){\line(0,-1){90}}
\put(100,105){$a$}
\put(130,55){$b$}
\put(93,55){$\Pi_{a,b}$}
%\put(211,12){\small{$ \varphi$}}
\end{picture}
\end{center}
\caption{A rectangle $\Pi_{a,b}\subset\dR^2$ with the edges of lengths $a,b>0$.}
\label{fig2}
\end{figure}
We show how simple one can estimate $\cE(\Pi_{a,b})$ from below using our methods. First we collect and prove some properties of the functions $F_{\Pi_{a,b}}$ and $F_{\dR^2\setminus\ov{\Pi_{a,b}}}$.
\begin{lem}
\label{lem:FP}
Let $\Pi_{a,b}$ be a rectangle as in \eqref{rectangle} with $a, b>0$. 
Let the functions $F_{\Pi_{a,b}}$ and $F_{\dR^2\setminus\ov{\Pi_{a,b}}}$ be defined as in \eqref{fOmega}. Then the following statements hold:
\begin{itemize}\setlength{\parskip}{1.2mm}
\item [\rm (i)] $F_{\Pi_{a,b}}(\beta) \le -\beta\big(\frac{2}{a}+\frac{2}{b}\big)$;
\item [\rm (ii)] $F_{\dR^2\setminus\ov{\Pi_{a,b}}}(\beta) \ge -\beta^2$.
\end{itemize}
\end{lem}
\begin{proof}
In the proof we convent to write $\Pi$ instead of $\Pi_{a,b}$.

\noindent (i) Let us take the characteristic function $\chi_{\Pi}$
of the domain $\Pi$. Clearly, we have $\chi_{\Pi}\in H^1(\Pi)$ and
\[
\frt_\beta^{\Pi}[\chi_{\Pi},\chi_{\Pi}] = -\beta(2a+2b),
\] 
where we used that $\nabla \chi_{\Pi} = 0$ and that $|\partial \Pi| = 2a+ 2b$. Note that $\|\chi_{\Pi}\|_{L^2}^2 = ab$ and that
\[
\inf\sigma(-\Delta_\beta^\Pi) \le \frac{\frt_{\beta}^\Pi[\chi_{\Pi},\chi_\Pi]}{\|\chi_\Pi\|_{L^2(\dR^2)}^2} \le -\beta\big(\tfrac{2}{a}+\tfrac{2}{b}\big),
\]
and the claim is proven.

\noindent (ii) Let us split the domain $\dR^2\setminus\ov{\Pi}$ into the partition $\cP = \{\Omega_k\}_{k=1}^8$ as shown on Figure~\ref{fig3}.
\begin{figure}[H]
\begin{center}
\begin{picture}(200,200)
\put(75,150){\line(1,0){50}}
\put(75,150){\line(1,0){50}}
\put(75,60){\line(1,0){50}}
\put(75,150){\line(0,-1){90}}
\put(125,150){\line(0,-1){90}}
%\put(100,155){$a$}
%\put(130,105){$b$}
\put(97,105){$\Pi$}
\put(97,170){$\Omega_1$}
\put(147,100){$\Omega_2$}
\put(97,40){$\Omega_3$}
\put(52,100){$\Omega_4$}
\put(52,170){$\Omega_5$}
\put(52,40){$\Omega_8$}
\put(147,170){$\Omega_6$}
\put(147,40){$\Omega_7$}

\multiput(75, 150)(0, 5 ){10}{\line(0,1){2}}
\multiput(125, 150)(0, 5 ){10}{\line(0,1){2}}
\multiput(75, 60)(0, -5 ){10}{\line(0,-1){2}}
\multiput(125, 60)(0, -5 ){10}{\line(0,-1){2}}
\multiput(75, 150)(-5, 0 ){10}{\line(-1,0){2}}
\multiput(125, 150)(5, 0 ){10}{\line(1,0){2}}
\multiput(75, 60)(-5, 0 ){10}{\line(-1,0){2}}
\multiput(125, 60)(5, 0 ){10}{\line(1,0){2}}

%\put(211,12){\small{$ \varphi$}}
\end{picture}
\end{center}
\caption{Partition $\cP = \{\Omega_k\}_{k=1}^8$ of $\dR^2\setminus\ov{\Pi}$ }
\label{fig3}
\end{figure}
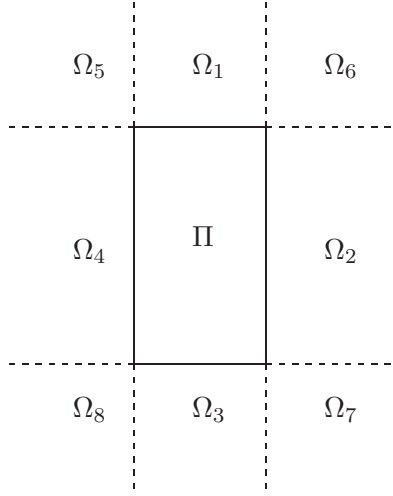
\noindent We use the notation $f_k := f|_{\Omega_k}$. Consider the sesquilinear form
\[
\begin{split}
\frt^{\dR^2\setminus{\ov \Pi}}_{\beta,\rm N}[f,g] &:= \sum_{k=1}^8(\nabla f_k,\nabla g_k)_{L^2(\Omega_k;\dC^2)} - \beta(f|_{\partial\Pi}, g|_{\partial\Pi})_{L^2(\partial\Pi)},\\ 
\dom  \frt^{\dR^2\setminus{\ov \Pi}}_{\beta,\rm N} &:= \bigoplus_{k=1}^8 H^1(\Omega_k).
\end{split}
\]
The form above is clearly, closed, densely defined, lower-semibounded and symmetric. It generates a self-adjoint operator $-\Delta^{\dR^2\setminus{\ov \Pi}}_{\beta, \rm N}$, which is an orthogonal sum of $8$ self-adjoint operators acting in $L^2(\Omega_k)$ with $k=1,2\dots,8$, respectively. The spectra of the components  corresponding to $\Omega_5$, $\Omega_6$, $\Omega_7$ and $\Omega_8$ are equal to $[0,+\infty)$. The spectra of the components corresponding to $\Omega_1$, $\Omega_2$, $\Omega_3$ and $\Omega_4$ are equal to $[-\beta^2,+\infty)$, which can be seen from separation of variables on these domains.
Hence, we get
\[
\sigma(-\Delta^{\dR^2\setminus{\ov \Pi}}_{\beta, \rm N}) = [-\beta^2,+\infty).
\]
Note that the ordering
\[
\frt^{\dR^2\setminus{\ov \Pi}}_{\beta,\rm N}\subset \frt^{\dR^2\setminus{\ov \Pi}}_{\beta}
\]
holds in the sense of \cite[\S VI.5]{Kato} and of \cite[\S 10.2]{BS87} and the estimate
\[
F_{\dR^2\setminus{\ov \Pi}}(\beta) =\inf\sigma(-\Delta_\beta^{\dR^2\setminus\ov{\Pi}}) \ge\inf\sigma(-\Delta_{\beta, \rm N}^{\dR^2\setminus\ov{\Pi}}) =-\beta^2
\]
follows with the help of \cite[\S 10.2, Theorem 4]{BS87}.
\end{proof}
\begin{prop}
Let $\Pi_{a,b}$ be a rectangle as in \eqref{rectangle}, and let the value $\cE(\Pi_{a,b})$ be defined as in \eqref{alpha1}.
Then the following estimate
\[
\cE(\Pi_{a,b}) \ge \sqrt{1 + 2\Big(\frac{1}{a}+\frac{1}{b}\Big)}
\]
holds.
\end{prop}
\begin{proof}
From monotonicity of $F_{\Pi_{a,b}}$ proven in Proposition~\ref{prop:f}\,(i) and from Lemma~\ref{lem:FP}\,(i) we obtain that
\begin{equation}
\label{beta3}
\beta(\Pi_{a,b}) \le \frac{ab}{2a+2b}.
\end{equation}
Analogously, we get with the aid of Lemma~\ref{lem:FP}\,(ii) that
\begin{equation}
\label{beta4}
\beta(\dR^2\setminus\ov{\Pi_{a,b}}) \ge 1.
\end{equation}
Using Theorem~\ref{thm:bnd1} and estimates \eqref{beta3}, \eqref{beta4}  we arrive at
\[
\cE(\Pi_{a,b}) \ge \sqrt{1+ \frac{\beta(\dR^2\setminus\ov{\Pi_{a,b}})}{\beta(\Pi_{a,b})}} \ge\sqrt{1 + 2\Big(\frac{1}{a}+\frac{1}{b}\Big)},
\]
that finishes the proof.
\end{proof}
\begin{remark}
The method given in this subsection extends easily to parallelepipeds
\[
\Pi_{a_1,a_2,\dots, a_d} = (0,a_1)\times(0,a_2)\times\dots\times(0,a_d)\subset\dR^d.
\]
\end{remark}

\end{document}